\documentclass[12pt,oneside,english]{amsart}
\usepackage[T1]{fontenc}
\usepackage[utf8]{inputenc}
\usepackage{geometry}
\usepackage{babel}
\usepackage{url}
\usepackage{amstext}
\usepackage{amsthm}
\usepackage{amssymb}
\usepackage[unicode=true,pdfusetitle,
 bookmarks=true,bookmarksnumbered=false,bookmarksopen=false,
 breaklinks=false,pdfborder={0 0 1},backref=false,colorlinks=false]
 {hyperref}

\makeatletter
\numberwithin{equation}{section}
\numberwithin{figure}{section}
\theoremstyle{plain}
\newtheorem{thm}{\protect\theoremname}[section]
\theoremstyle{remark}
\newtheorem{notation}[thm]{\protect\notationname}
\theoremstyle{plain}
\newtheorem{cor}[thm]{\protect\corollaryname}
\theoremstyle{plain}
\newtheorem{lem}[thm]{\protect\lemmaname}
\theoremstyle{definition}
\newtheorem{defn}[thm]{\protect\definitionname}
\newtheorem{example}[thm]{\protect\examplename}
\theoremstyle{plain}
\newtheorem{prop}[thm]{\protect\propositionname}
\theoremstyle{remark}
\newtheorem{rem}[thm]{\protect\remarkname}
\theoremstyle{plain}
\newtheorem{lyxalgorithm}[thm]{\protect\algorithmname}

\usepackage{lmodern}

\makeatother

\providecommand{\algorithmname}{Algorithm}
\providecommand{\corollaryname}{Corollary}
\providecommand{\definitionname}{Definition}
\providecommand{\examplename}{Example}
\providecommand{\lemmaname}{Lemma}
\providecommand{\notationname}{Notation}
\providecommand{\propositionname}{Proposition}
\providecommand{\remarkname}{Remark}
\providecommand{\theoremname}{Theorem}

\begin{document}
\global\long\def\cc{\mathbb{C}}%

\global\long\def\ff{\mathbb{F}}%

\global\long\def\pp{\mathbb{P}}%

\global\long\def\rr{\mathbb{R}}%

\global\long\def\zz{\mathbb{Z}}%

\global\long\def\tensor{\otimes}%

\global\long\def\extalg{\bigwedge}%

\global\long\def\spanop{\operatorname{span}}%

\global\long\def\eval{\operatorname{eval}}%

\global\long\def\inc{\operatorname{inc}}%

\global\long\def\shift#1#2{\operatorname{shift}_{#1\to#2}}%

\global\long\def\slowshift#1#2{\mathop{N_{#1\to#2}}}%

\global\long\def\grassman#1#2{\operatorname{Gr}(#1,#2)}%

\title{A Hilton-Milner theorem for exterior algebras}
\author{Denys Bulavka, Francesca Gandini, and Russ Woodroofe}
\thanks{Work started while the first author was at the Department of Applied Mathematics, Faculty of Mathematics and Physics, Charles University, Prague, Czech Republic. Work of the first author is partially supported by the GAČR grant no. 22-19073S and by the Israel Science Foundation grant ISF-2480/20.  Work of the second author is supported in part by the Slovenian Research Agency (research projects N1-0160, J1-3003). Work of the third author is supported in part by the Slovenian Research
Agency research program P1-0285 and research projects J1-9108, N1-0160, J1-2451, J1-3003, and J1-50000.}
\address{Einstein Institute of Mathematics, Hebrew University, Jerusalem 91904,
Israel}
\curraddr{Department of Mathematics \& Statistics, Dalhousie University, 6297 Castine Way, PO BOX 15000, Halifax, NS, Canada, B3H 4R2}
\email{denys.bulavka@dal.ca}
\urladdr{\url{https://kam.mff.cuni.cz/~dbulavka/}}
\address{Department of Mathematics, Statistics, and Computer Science, St. Olaf
College, Northfield MN, USA}
\email{fra.gandi.phd@gmail.com}
\urladdr{\url{https://sites.google.com/a/umich.edu/gandini/}}
\address{Univerza na Primorskem, Glagoljaška 8, 6000 Koper, Slovenia}
\email{russ.woodroofe@famnit.upr.si}
\urladdr{\url{https://osebje.famnit.upr.si/~russ.woodroofe/}}
\begin{abstract}
Recent work of Scott and Wilmer and of Woodroofe extends the Erd\H{o}s-Ko-Rado
theorem from set systems to subspaces of $k$-forms in an exterior
algebra. We prove an extension of the Hilton-Milner theorem to the
exterior algebra setting, answering in a strong way a question asked
by these authors.
\end{abstract}

\maketitle

\section{\label{sec:Introduction}Introduction}

A family of sets $\mathcal{F}$ is \emph{pairwise-intersecting} if
every pair of sets in $\mathcal{F}$ have nonempty intersection. Erd\H{o}s,
Ko, and Rado \cite{Erdos/Ko/Rado:1961} gave an upper bound on the
size of a pairwise-intersecting family of small sets, and characterized
the families that achieve the upper bound.
\begin{thm}[Erd\H{o}s-Ko-Rado]
\label{thm:EKR} Let $k\leq n/2$. If $\mathcal{F}\subseteq \binom{[n]}{k}$
is a pairwise-intersecting family of sets, then $\left|\mathcal{F}\right|\leq \binom{n-1}{k-1}$.
Moreover, if $k<n/2$ and $\left|\mathcal{F}\right|$ achieves the
upper bound, then $\mathcal{F}$ consists of all the $k$-subsets
containing some fixed element.
\end{thm}

There are many generalizations and analogues of Theorem~\ref{thm:EKR}.
In this article, we focus on the extension to exterior algebras that
was considered by Scott and Wilmer \cite{Scott/Wilmer:2021} and by
the third author \cite{Woodroofe:2022}. We first fix some notation:
\begin{notation}
\label{not:FixedNotation}
Through the paper, let $\ff$ be a field.  We assume for expository purposes that the characteristic is not
$2$, although the fundamental techniques are independent of characteristic. Let $V$ be a vector space of dimension $n$ over $\ff$. Let $\mathbf{e}=\left\{ e_{1},\dots,e_{n}\right\} $
be a given basis for $V$. If $I$ is a subset of $[n]$, then let
$V^{(I)}$ be the subspace spanned by $\{e_{j}:j\notin I\}$, and
write $V^{(i)}$ for $V^{(\{i\})}$.
\end{notation}

A subset $L$ of the exterior algebra is \emph{self-annihilating}
if $L\wedge L=0$; that is, if for every $x,y\in L$, it holds that
$x\wedge y=0$.
\begin{thm}[Scott and Wilmer]
\label{thm:EKRext}Let $k\leq n/2$. If $L$ is a self-annihilating
subspace of $\extalg^{k}V$, then $\dim L\leq\binom{n-1}{k-1}$. 
\end{thm}

The bound of Theorem~\ref{thm:EKRext} is an extension of the bound
of Theorem~\ref{thm:EKR}, as follows. Given a set $F=\{i_{1}<i_{2}<\cdots<i_{k}\}$
in $\mathcal{F}$, we represent it with the exterior monomial $e_{F}=e_{i_{1}}\wedge\cdots\wedge e_{i_{k}}$
in $\extalg^{k}V$. Now taking $\{e_{F}:F\in\mathcal{F}\}$ as a basis,
using that $e_{F}\wedge e_{G}=0$ if and only if $F\cap G\neq\emptyset$,
and applying the distributive axiom, we obtain the bound of Theorem~\ref{thm:EKR}
from Theorem~\ref{thm:EKRext}. With this tight connection, Theorem~\ref{thm:EKRext}
may indeed be viewed as a categorification of the Erdős-Ko-Rado bound.

Scott and Wilmer asked in \cite[Section 2.2]{Scott/Wilmer:2021} whether the characterization of the maximal
families in Theorem~\ref{thm:EKR} extends to the exterior algebra
setting. The third author
of the current article asked in \cite[Question 5.1]{Woodroofe:2022}
the even stronger question of whether the Hilton-Milner upper bound
extends. We will answer both questions in the affirmative. 

We first recall the theorem of Hilton and Milner \cite{Hilton/Milner:1967}.
We say that a pairwise-intersecting set family is \emph{nontrivial}
if $\bigcap_{F\in\mathcal{F}}F$ is empty.
\begin{thm}[Hilton-Milner]
\label{thm:HM}Let $k\leq n/2$. If $\mathcal{F}\subseteq\binom{[n]}{k}$
is a nontrivially pairwise-intersecting family of sets, then $\left|\mathcal{F}\right|\leq\binom{n-1}{k-1}-\binom{n-k-1}{k-1}+1$. 
\end{thm}

Our main theorem will extend Theorem~\ref{thm:HM} to exterior algebras.
We say that a self-annihilating subset $L$ of $\extalg V$ is \emph{nontrivial}
if there is no $1$-form in $V=\extalg^{1}V$ that annihilates $L$.
\begin{thm}
\label{thm:HMext} Let $k\leq n/2$. If $L$ is a nontrivially self-annihilating
subspace of $\extalg^{k}V$, then $\dim L\leq\binom{n-1}{k-1}-\binom{n-k-1}{k-1}+1$.
\end{thm}

Theorem~\ref{thm:HMext} categorifies Theorem~\ref{thm:HM} in the
same way that Theorem~\ref{thm:EKRext} categorifies the bound of Theorem~\ref{thm:EKR}.
Since $\binom{n-k-1}{k-1}>1$ when $k<n/2$, we recover the exterior
extension of the characterization of maximal families in Theorem~\ref{thm:EKR}.
\begin{cor}
Let $k<n/2$. If $L$ is a self-annihilating subspace of $\extalg^{k}V$
such that $\dim L$ achieves the upper bound of $\binom{n-1}{k-1}$,
then $L$ is annihilated by a $1$-form.
\end{cor}

Theorem~\ref{thm:HMext} was earlier shown in \cite{Ivanov/Kose:2023}
for the restrictive case where $k=2$. 

Our strategy is to adapt proofs of Theorem~\ref{thm:HM}, variously
due to Frankl and/or Füredi \cite{Frankl:1987,Frankl:2019,Frankl/Furedi:1986}
and to Hurlbert and Kamat \cite{Hurlbert/Kamat:2018}, to the exterior
algebra setting. These proofs start with a nontrivially pairwise-intersecting
set family, and apply combinatorial shifting operations. If the family
becomes trivial at some step, then this guarantees strong structural
properties at the previous step. Wise choices of further combinatorial
shifting operations results in a nontrivially pairwise-intersecting
family that is of the same size and shifted.

In order to implement this strategy in the exterior algebra setting, we take limits of a parameterized family of linear maps. Draisma, Kraft, and Kuttler used similar techniques in an earlier paper \cite{Draisma/Kraft/Kuttler:2006} on the Gerstenhaber problem, concerning maximal nilpotent subspaces of a Lie algebra. In this paper, we introduce a parameterized family of linear maps acting on the exterior algebra. The limits of these actions extend combinatorial shifting to exterior algebras, and preserve more structure than those used by the third author in \cite[Section 4]{Woodroofe:2022}. These limit actions allow us to follow the general outline of the set system proof, although significant additional care is needed in the broader (non-monomial) exterior algebra situation.  

Our techniques are characteristic independent; and while they are based on ideas from algebraic geometry, our presentation here is self-contained and relatively elementary.  From the algebraic geometry perspective, a key difficulty arises from the fact that the condition of not being annihilated by any $1$-form fails to be Zariski closed.  This failure is parallel to the fact that the shift of a nontrivially pairwise-intersecting set system may become trivial.

In addition to an analogue of Theorem~\ref{thm:HM}, we also prove
an extension of another theorem of Hilton and Milner. Set families
$\mathcal{F}$ and $\mathcal{G}$ are \emph{cross-intersecting} if
for every $F\in\mathcal{F}$ and $G\in\mathcal{G}$, the intersection
$F\cap G$ is nonempty. Hilton and Milner gave bounds on the size
of $\left|\mathcal{F}\right|+\left|\mathcal{G}\right|$ under the
conditions that $\mathcal{F},\mathcal{G}$ are nonempty and cross-intersecting
(as a special case of a more general but technical result \cite[Theorem 2]{Hilton/Milner:1967}).

The exterior algebra analogue is as follows. Subsets $K$ and $L$
of $\extalg V$ are \emph{cross-annihilating} if $K\wedge L=0$, that
is, if for every $x\in K$ and $y\in L$ it holds that $x\wedge y=0$.  Thus, a subset is self-annihilating exactly if it is cross-annihilating with itself.
We show:
\begin{thm}
\label{thm:CrossAnnihilating-ext}Let $k\leq n/2$. If $K$ and $L$
are nonzero cross-annihilating subspaces of $\extalg^{k}V$, then $\dim K+\dim L\leq\binom{n}{k}-\binom{n-k}{k}+1$.
\end{thm}

A different bound on cross-annihilating subspaces was given by Scott and Wilmer in \cite[Theorem 2.4]{Scott/Wilmer:2021}, where they used initial monomial techniques to bound the product of $\dim K$ and $\dim L$.

We have written the paper to be accessible to an extremal combinatorialist, who we do not assume to be well-versed in algebraic geometry.  We now give a brief overview directed towards the interested algebraic geometer.  The main geometric object is a subvariety of a Grassmannian over the space of $k$-forms in the exterior algebra.   Intersecting families of sets correspond to the points (subspaces) admitting a monomial basis.  Combinatorial shifting and generalizations corresponds to limits taken over certain algebraic curves between points.  Thus, we place extremal set theory problems in a moduli space setting.  The technical difficulty in Theorem~\ref{thm:HMext} lies in the previously-mentioned fact that the ``no annihilating 1-form'' condition is not Zariski closed, and care is needed to guarantee that a curve (and its limit point) remains in the subset.
\smallskip

The rest of this article is organized as follows. In Section~\ref{sec:Preliminaries}
we give various preliminaries and background results. In Section~\ref{sec:SlowerShifting}
we introduce the linear maps whose limits will perform shifting-like operations in the exterior setting. In Section~\ref{sec:CrossAnn}
we prove Theorem~\ref{thm:CrossAnnihilating-ext}, and in Section~\ref{sec:HMproof}
we prove Theorem~\ref{thm:HMext}.

\section*{Acknowledgements}

We thank Allen Knutson and Jake Levinson for clarifying our doubts regarding limits
on the Grassmannian and  in the broader sense of algebraic geometry.

\section{\label{sec:Preliminaries}Preliminaries}

\subsection{Combinatorial shifting and shifted families of sets}
\label{subsec:CombShift}
Let $\mathcal{F}\subseteq\binom{[n]}{k}$ be a family of sets, let
$F\in\mathcal{F}$, and let $1\leq i<j\leq n$. The \emph{combinatorial
shift} operation $\shift ji$ is defined as follows.

\begin{alignat*}{1}
\shift ji(F,\mathcal{F}) & =\begin{cases}
(F\setminus\{j\})\cup\{i\} & \text{if }j\in F,i\notin F\text{ and }(F\setminus\{j\})\cup\{i\}\notin\mathcal{F},\\
F & \text{otherwise.}
\end{cases}\\[3pt]
\shift ji\mathcal{F} & =\{\shift ji(F,\mathcal{F})\ \colon\ F\in\mathcal{F}\}.
\end{alignat*}

For $I\subseteq[n]$, we say that a family $\mathcal{F}$ of subsets
of $[n]$ is \emph{shifted with respect to $I$} if $\shift ji\mathcal{F}=\mathcal{F}$
for each pair $i<j$ with $i,j\in I$. If $I=[n]$, we simply say
that $\mathcal{F}$ is \emph{shifted}. It is well-known that iteratively
applying combinatorial shift operations results in a shifted family
\cite[Section 2]{Frankl:1987}; see also Theorem~\ref{thm:ShiftingTerminates}
below.

The technique of combinatorial shifting was introduced by Erdős,
Ko, and Rado in \cite{Erdos/Ko/Rado:1961} to prove Theorem~\ref{thm:EKR}.
It has since become a standard tool in combinatorial set theory \cite{Frankl:1987}.

\subsection{Exterior algebras}

The \emph{exterior algebra} over the vector space $V$, denoted $\extalg V$, is a
graded algebra that is described as follows. For each $0\leq k\leq n$,
we define the $k$th graded component $\extalg^{k}V$ as the span
of elements $\left\{ e_{S}:S\in\binom{[n]}{k}\right\} $. Here, we
identify $\extalg^{1}V$ with $V$ by identifying $e_{\{i\}}$ with
$e_{i}$; we identify $\extalg^{0}V$ with $\ff$. The elements of
$\extalg^{k}V$ are called \emph{$k$-forms}.

Now the product is defined to satisfy the following axioms: $e_{\emptyset}\wedge e_{S}=1e_{S}=e_{S}$,
$e_{i}\wedge e_{j}=-e_{j}\wedge e_{i}$, and $e_{S}=e_{s_{1}}\wedge e_{s_{2}}\wedge\cdots\wedge e_{s_{k}}$
for $S=\left\{ s_{1}<s_{2}<\cdots<s_{k}\right\} $. The elements $e_{S}$
are \emph{monomials with respect to the basis }$\mathbf{e}$. 

The exterior algebra is an antisymmetric analogue of the polynomial
algebra, and more background may be found in most advanced algebra textbooks, such as \cite{Dummit/Foote:2004}.  It has long been used as a tool for proving results in combinatorial
set theory \cite{Babai/Frankl:1992}. 

Given a $k$-form $x$ and an understood basis, the \emph{support} of $x$ is the set of monomials with nonzero coefficient in expression for $x$ as a linear combination of monomials.  The monomial $e_S$ has \emph{variable set} $\{ e_i : i\in S \}$.  When it causes no confusion, we may identify a variable set $\{ e_i : i\in S\}$ with the underlying set of indices $S$.

Given a linear operator $M$ on $V$, we extend the linear operator
to $\extalg V$ by sending $e_{S}=e_{s_{1}}\wedge e_{s_{2}}\wedge\cdots\wedge e_{s_{k}}$
to $Me_{s_{1}}\wedge Me_{s_{2}}\wedge\cdots\wedge Me_{s_{k}}$ and
extending linearly. We notice that in particular, for any $x,y$ in
$\extalg V$, we have that $M(x\wedge y)=Mx\wedge My$.

The following result on annihilation by $1$-forms is well-known;
see e.g. \cite[Lecture 6]{Harris:1995}.
\begin{lem}
\label{lem:1formAnn} For $v\in\extalg^{1}V$, $x\in\extalg^{k}V$,
we have that $v\wedge x=0$ if and only if $x=v\wedge x'$ for some
$x'\in\extalg^{k-1}V$.
\end{lem}

\subsection{\label{subsec:LimitActions}Limit actions on the Grassmannian}

Write $\pp V$ for the projective space defined on the vector space $V$, consisting of (equivalence classes of) vectors of $V\setminus\{0\}$, considered up to scalar multiplication. If $v_{1},v_{2},\dots,v_{k}$
are linearly independent vectors in $V$, then we may identify the subspace spanned by these vectors with the
point $[v_{1}\wedge v_{2}\wedge\cdots\wedge v_{k}]$ in $\pp(\extalg^{k}V)$. We get the \emph{Grassmannian}
$\grassman kV$, consisting of all points in $\pp(\extalg^{k}V)$
that may be written in this form. Thus, the Grassmannian is a geometric
object whose points correspond to $k$-dimensional subspaces of $V$.

We will need to take limits as $t\to0$ of actions of families of matrices parameterized 
by $t$. The limits we will take are basically of an algebraic geometry
nature, but we did not find accessible literature on the topic. 
For
completeness, we give an elementary description, which we hope will be accessible to any reader who has had a solid graduate course in (linear) algebra.  

We start by describing the action on an arbitrary projective space
$\pp V$. Next, extend our field $\ff$ to $\ff(t)$ by adjoining a transcendental
$t$. Thus, $\ff(t)$ consists of ratios of polynomials in $t$ with
coefficients in $\ff$. 

We now extend coefficients in our vector space $V$ over $\ff$ to
a vector space $V(t)\cong V\tensor_{\ff}\ff(t)$ over $\ff(t)$. Thus,
an element of $V(t)$ is a linear combination of $e_{1},\dots,e_{n}$
with coefficients in $\ff(t)$. By clearing denominators and/or dividing
by $t$ in $\pp V(t)$, we may write an arbitrary projective element
$v(t)$ as a combination of $e_{1},\dots,e_{n}$, where each coefficient
is a polynomial in $t$, and where the coefficients do not have a
common divisor polynomial of positive $t$-degree. This form is the \emph{canonical
representative of $v(t)$}. In particular, the (polynomial) coefficient
of at least one $e_{i}$ has a nonzero coefficient of the $t$-degree
zero term. Now the map $\eval_{0}:\pp V(t)\to\pp V$ obtained by evaluating
the polynomials in the canonical representative at $t=0$ is a well-defined
map.

Now let $N(t)$ be a nonsingular linear map over $\ff(t)$.
We define the \emph{limit} of the action of $N(t)$ on $\pp V$ to
be the composition 
\[
\pp V\xrightarrow{\iota}\pp V(t)\xrightarrow{N(t)}\pp V(t)\xrightarrow{\eval_{0}}\pp V,
\]
where $\iota$ is the standard inclusion map. We write $N$ for the
resulting map $\pp V\to\pp V$.
\begin{example}
For the limit of $N(t)$ acting on any fixed vector $v$ of $\pp\rr^{2}$, we can write $\mathop{N(t)}v = w$ in coordinates as $[w_1(t), w_2(t)]$.  This point corresponds to the formal ratio $w_1(t)/ w_2(t)$. Our description of the limit map resembles closely the usual procedure for evaluating
$\lim_{t\to0}w_{1}(t)/w_{2}(t)$ from a first calculus course. Note that here the limit point $[\alpha, 1]$ corresponds to the real number $\alpha$, while the (single) limit point $[1,0]$ corresponds to the calculus limits $\pm \infty$.
\end{example}

The action on $\pp(\extalg^{k}V)$ is now a special case: the basis
for $V$ gives a basis of monomials for $\extalg^{k}V$, and a linear
map $N(t)$ on $V(t)$ induces a linear map on $\extalg^{k}V(t)$.
Similarly for $\pp(\extalg^{r}(\extalg^{k}V))$, given a nonnegative integer $r$.

We are ready to discuss the limit of the action on $\grassman rV$ and $\grassman r{\extalg^{k}V}$. 
\begin{prop}
Let $N(t)$ be a nonsingular linear map $V(t)\to V(t)$.
If $L=v_{1}\wedge\cdots\wedge v_{r}$ is a point on $\grassman rV\subseteq\pp(\extalg^{r}V)$,
then the limit $NL$ is also in $\grassman rV$, and is the vector
subspace spanned by $\{Nw:w\in L\}$.
\end{prop}

\begin{proof}
We first notice that a vector $w$ is in $L$ if and only if $w\wedge L=0$.
Now $\mathop{N(t)}(w\wedge L)=\mathop{N(t)}w\wedge \mathop{N(t)}L=0$, and so the same happens
in $t$-degree zero in the canonical representative. It follows that
if $w\wedge L=0$, then $Nw\wedge NL=0$.

Conversely, we need to show that $NL$ is annihilated by $r$ linearly
independent vectors in the image of $N$. Suppose that $e_{i_{1}}\wedge\cdots\wedge e_{i_{r}}$
is in the support of $NL$. Then $e_{i_{1}}\wedge\cdots\wedge e_{i_{r}}$
has a $t$-degree zero component in the canonical representative of
$\mathop{N(t)}L$. The coefficient of this monomial is an $r\times r$ minor
of the the matrix with columns $\mathop{N(t)}v_{1},\dots,\mathop{N(t)}v_{r}$. By Gaussian
elimination, we may find vectors $\mathop{N(t)}w_{1},\dots,\mathop{N(t)}w_{r}$ in $\mathop{N(t)}L$
so that the corresponding minor of $\mathop{N(t)}w_{1},\dots,\mathop{N(t)}w_{r}$ is
diagonal and has a $t$-degree zero component. It follows that $Nw_{1},\dots,Nw_{r}$
are linearly independent, as desired.
\end{proof}
\begin{rem}
Equivalently, if the matrix $W$ with columns $\mathop{N(t)}v_{1},\dots,\mathop{N(t)}v_{r}$
has an $r\times r$ minor with a $t$-degree zero component, then
let $G$ be the corresponding $r\times r$ submatrix. Now the columns
of $WG^{-1}$ are vectors in $\mathop{N(t)}L$ whose limits are linearly independent.
\end{rem}

\begin{cor}
Let $N(t)$ be a nonsingular linear map $V(t)\to V(t)$.
If $L_{0}\subseteq L$ are vector subspaces of $V$, then also $NL_{0}\subseteq NL$.
\end{cor}

As before, the limit action $N$ on $\grassman r{\extalg^{k}V}$ is
a special case.

For the purpose of this paper, we need that the limit preserves the self-annihilating and cross-annihilating
properties. Experts in algebraic geometry will recognize this as following
from the fact that these properties correspond to closed sets in the
Zariski topology. An elementary proof is also easy.
\begin{lem}
\label{lem:LimitsPreserveCrossAnn}Let $N(t)$ be a nonsingular
linear map $V(t)\to V(t)$. If $L$ and $L'$ are cross-annihilating
subspaces of $\extalg V$, then $NL$ and $NL'$ are also cross-annihilating.
\end{lem}

\begin{proof}
Let $x$ be in $L$ and $x'$ in $L'$. The $t$-degree zero term
of the exterior product of the canonical representatives of $\mathop{N(t)}x$
and $\mathop{N(t)}x'$ is the product of the $t$-degree zero terms of $\mathop{N(t)}x$
and $\mathop{N(t)}x'$. Now if $x\wedge x'=0$, so that also $\mathop{N(t)}x\wedge \mathop{N(t)}x'=0$,
then we recover that $Nx\wedge Nx'=0$.
\end{proof}

For further reading on projective space and the Grassmannian, we refer the reader to \cite{Artin:2022,Harris:1995}, for example.  Accessible literature on limits in algebraic geometry (also known as \emph{degenerations} or \emph{specializations}) seems to be a bit difficult to find.  Artin \cite[Chapter 5]{Artin:2022} discusses limits of curves in algebraic geometry briefly, mainly over the field of complex numbers.  The description we have given is known (in a somewhat more abstract setting) as the \emph{valuative criterion}, as discussed by Newstead in \cite[Preliminaries]{Newstead:1978}.  Eisenbud and Harris \cite[II.3]{Eisenbud/Harris:2000} also give a somewhat similar description to ours in a more abstract and general situation.

\section{\label{sec:SlowerShifting}Slower shifting}

In this section, we describe a family of linear operators and their limits that generalize combinatorial shifting.  The operators are similar to but distinct from  those given by the third author in \cite[Section 4.1]{Woodroofe:2022}, and earlier by Knutson in \cite{Knutson:2014UNP}.  Murai and Hibi had previously examined these operators in \cite[Section 2]{Murai/Hibi:2009} (see also \cite[Section 11.3.2]{Herzog/Hibi:2011}) applied to exterior ideals with a basis of monomials.  
In the case where a subspace of the exterior algebra does not admit a basis of monomials,
our matrices will preserve more structure in comparison to those of \cite{Knutson:2014UNP,Woodroofe:2022}, yielding a slower and gentler shifting operation.

Let $i$ and $j$ be between $1$ and $n$. As before, let $\mathbb{F}(t)$
be the field extension with a new transcendental element $t$, and
let $V(t)\cong V\tensor\ff(t)$ be the extension of $V$ to coefficients
in $\ff(t)$. 

\begin{defn}    
Let $N_{j\to i}(t)$ be the linear map sending $e_{j}\mapsto e_{i}+te_{j}$,
and fixing all other basis elements. 
The limit of the action of $N_{j\to i}(t)$ is the \emph{slow shifting} operation.


\end{defn}

Also of some interest is the linear map $M_{i}(t)$, sending $e_{i}\mapsto te_{i}$
and fixing all other $e_{j}$. Obviously, both $N_{j \to i}(t)$ and $M_{i}(t)$
depend on the choice of (ordered) basis $\mathbf{e}$.  In particular, as matrices acting on the left, we have 
\[ \mathop{N_{j\to i}(t)}=\bordermatrix{       &   & \scriptstyle{i} &   & \scriptstyle{j} & \cr       & 1 &   &   &   & \cr     \scriptstyle{i} &   & 1 &   & 1 & \cr       &   &   & \ddots   &   &\cr     \scriptstyle{j} &   &   &   & t & \cr       &   &   &   &   & 1 \cr},  \qquad M_i(t) =    \bordermatrix{      &   &        &\scriptstyle{i}   &        &   \cr      & 1 &        &    &        &   \cr      &   & \ddots &    &        &   \cr     \scriptstyle{i}&   &        &t   &        &   \cr      &   &        &    & \ddots &   \cr      &   &        &    &        & 1 \cr}. \]
As in Section~\ref{subsec:LimitActions}, we may extend an action
on $V(t)$ to the Grassmannian, and consider the limit action obtained
by evaluating the canonical representative at $t=0$. Denote the limit
of $N_{j\to i}(t)$ as $\slowshift ji$, and that of $M_{i}(t)$ as $M_{i}$.

\subsection{The action of \texorpdfstring{$\slowshift ji$}{Nij} }
\label{subsec:SlowShift}
We first describe the action of $\slowshift ji$ on (the projectivization
of) elements of $V$.  As in Notation~\ref{not:FixedNotation}, we denote by $V^{(j)}$ the vector subspace of $V$ spanned
by $\mathbf{e}\setminus\{e_{j}\}$.  Let $v=x+e_{j}$ be in $\pp V$, where
$x$ is in $V^{(j)}$.  Then $\mathop{N_{j\to i}(t)}v=x+e_{i}+te_{j}$. This
yields as canonical representative 
\[
\mathop{N_{j\to i}(t)}v=\begin{cases}
x+e_{i}+te_{j} & \text{if }e_{i}+x\neq0,\\
e_{j} & \text{otherwise.}
\end{cases}\text{\ensuremath{\qquad}Thus, }\slowshift ji v=\begin{cases}
x+e_{i} & \text{if }e_{i}+x\neq0,\\
e_{j} & \text{otherwise.}
\end{cases}
\]
We now extend this action to $\mathbb{P}(\extalg^{k}V)$. Here, if
$m=x+e_{j}\wedge y$, where the supports of $x$ and $y$ do not contain
$e_{j}$ in their variable sets, then $\mathop{N_{j\to i}(t)}m=x+e_{i}\wedge y+te_{j}\wedge y$. Similarly
to the above, this yields as canonical representative 
\begin{align}
\mathop{N_{j\to i}(t)}m & =\begin{cases}
x+e_{i}\wedge y+te_{j}\wedge y & \text{if }e_{i}\wedge y+x\neq0,\\
e_{j}\wedge y & \text{otherwise; and}
\end{cases}\label{eq:SlowShiftExtalg}\\
\slowshift ji m & =\begin{cases}
x+e_{i}\wedge y & \text{if }e_{i}\wedge y+x\neq0,\\
e_{j}\wedge y & \text{otherwise.}
\end{cases}\nonumber 
\end{align}
 We observe that when $m$ is a monomial with variable set
$F$, then $\slowshift ji m$ is the monomial with variable set $\shift{j}{i}F$.
More generally, if $L$ is generated by monomials whose variable sets form
the set system $\mathcal{F}$, then $\slowshift ji$ is generated by
the monomials with variable sets $\shift ji\mathcal{F}$. The latter statement
follows quickly from Lemma~\ref{lem:SlowShiftMonomialAtj} below,
or a direct proof with bases is also not difficult.

\subsection{Comparison with other operations}
\label{subsec:CompareOps}
Our operation $\slowshift ji$ slowly changes $e_{j}$'s to $e_{i}$'s
in an exterior $k$-form, while leaving alone monomials whose variable
sets avoid both. In contrast, the shifting operations considered in
\cite{Knutson:2014UNP,Woodroofe:2022} may quickly replace non-trivial systems with
trivial systems. Algebraic shifting \cite{Kalai:2002} or the related techniques
with initial monomials in \cite{Scott/Wilmer:2021} have similar defects.

Indeed, one essential ingredient of our proof is Lemma~\ref{lem:Technical2div} below,
which says that if the $1$-form $\ell$ annihilates a subspace of
$k$-forms after applying $\slowshift ji$, then the $2$-form $\ell\wedge(e_{i}-e_{j})$
annihilates before shifting. This is an exterior algebra analogue
of the fact that if $\shift{j}{i}\mathcal{F}$ becomes trivial, then before
shifting, every set in $\mathcal{F}$ contains at least one of $i$
or $j$. The analogue fails for the operations of \cite{Kalai:2002,Knutson:2014UNP,Scott/Wilmer:2021,Woodroofe:2022}.

To illustrate, we rephrase the operation of \cite{Knutson:2014UNP,Woodroofe:2022}
in the language of the current paper: it is the limit $O_{j\to i}$
as $t\to0$ of the linear map $O_{j\to i}(t)$ sending $e_{j}\mapsto e_{i}+te_{j}$
and $e_{h}\mapsto te_{h}$ for $h\neq j$. We consider the
$k$-form $z=e_{1}\wedge x+e_{2}\wedge y+w$, where $x$, $y$, and
$w$ are in $\extalg V^{(\{1,2\})}$. For $n$ large
enough relative to $k$, one can choose $x,y,w$ so that $z$ is not annihilated by any $2$-form.
Now $\slowshift 21 z$ is $e_{1}\wedge x+e_{1}\wedge y+w$. Depending on
the choice of $w$, this may not even be annihilated by any $2$-form, let alone a $1$-form.
In contrast, it is straightforward to see that $\mathop{O_{2\to1}}z$ is $e_{1}\wedge y$, which is annihilated
by $e_{1}$. We see that the analogue of Lemma~\ref{lem:Technical2div} for $O_{j\to i}$
does not hold, even for $1$-dimensional subspaces!

Similar drawbacks hold for initial monomial techniques, which replace $z$ with a monomial in a single step. 

\subsection{The action of \texorpdfstring{$M_{i}$}{Mi}}
\label{subsec:MinimizeVarAction}
Although it will be of less importance to us, it is not difficult
to describe the action of $M_{i}$. If $v=x+e_{i}$ is in $\pp V$
for $x\in V^{(i)}$, then 
\[
\mathop{M_{i}(t)}v=\begin{cases}
x+te_{i} & \text{if }x\neq0,\\
e_{i} & \text{otherwise.}
\end{cases}\text{\ensuremath{\qquad}Thus, }\mathop{M_{i}}v=\begin{cases}
x & \text{if }x\neq0,\\
e_{i} & \text{otherwise.}
\end{cases}
\]
 Extending this action to $\pp(\extalg^{k}V)$, we get for $m=x+e_{i}\wedge y$
(where the supports of $x$ and $y$ do not contain $e_{i}$ in their variable sets) the
following.

\[
\mathop{M_{i}(t)}m=\begin{cases}
x+te_{i}\wedge y & \text{if }x\neq0,\\
e_{i}\wedge y & \text{otherwise}.
\end{cases}\text{\ensuremath{\qquad}}\mathop{M_{i}}m=\begin{cases}
x & \text{if }x\neq0,\\
e_{i}\wedge y & \text{otherwise}.
\end{cases}
\]
 
\subsection{Stabilization of slow shifting}

We now show that a subspace $L$ of $\extalg^{k}V$ stabilizes under
repeated actions of our slow shifting $\slowshift ji$ operations.

A limit action of $\slowshift ji$ on $L$ is said to be \emph{fixing}
if $\slowshift ji L = L$; we will mainly be interested in non-fixing actions. Consider the following procedure.
\begin{lyxalgorithm}[Slow shifting procedure]
\label{alg:ShiftingAlg} Given $L$ a subspace of $\extalg^{k}V$,
and $I\subseteq[n]$:

While there exists a pair $i<j$ chosen from $I$ so that $\slowshift ji L\neq L$,

$\qquad$set $L:=\slowshift ji L$.

Return L.
\end{lyxalgorithm}

We will require some lemmas. We consider $\extalg V^{(j)}$
as a subalgebra of $\extalg V$. We say that $L$ is \emph{monomial
with respect to $e_{j}$} if 
\[
L=\left(L\cap\extalg^{k}V^{(j)}\right)\oplus\left(L\cap\big(e_{j}\wedge\extalg^{k-1}V^{(j)}\big)\right).
\]
That is, $L$ is monomial with respect to $e_{j}$ if we can find
a basis of elements that are either multiples of $e_{j}$, or else
do not have any monomials with $e_{j}$ in the support.  We also call such a basis \emph{monomial with respect to $e_{j}$}.
\begin{lem}
\label{lem:SlowShiftMakesMonomial}If $1\leq i<j\leq n$, then $\slowshift ji L$
is monomial with respect to $e_{j}$.
\end{lem}

\begin{proof}
Immediate by (\ref{eq:SlowShiftExtalg}).
\end{proof}

If $L$ is monomial with respect to $e_{h}$, then the slow
shift $\slowshift ji$ distributes over this monomiality for $h\neq i,j$, as follows.

\begin{lem}
\label{lem:SlowShiftMonomialAwayfromj}Let $L$ be a subspace of $\extalg^{k}V$,
and let $h\in[n]$. If $L$ is monomial with respect to $e_{h}$,
and $1\leq i<j\leq n$ are distinct from $h$, then 
\begin{align*}
\slowshift jiL & =\slowshift ji\left(L\cap\extalg^{k}V^{(h)}\right)\oplus \slowshift ji\left(L\cap\big(e_{h}\wedge\extalg^{k-1}V^{(h)}\big)\right)\\
 & =\left(\slowshift ji L\cap\extalg^{k}V^{(h)}\right)\oplus\left(\slowshift ji L\cap\big(e_{h}\wedge\extalg^{k-1}V^{(h)}\big)\right).
\end{align*}
\end{lem}

\begin{proof}
It is immediate by (\ref{eq:SlowShiftExtalg}) that 
\[
\slowshift ji \left(L\cap\extalg^{k}V^{(h)}\right)\subseteq\extalg^{k}V^{(h)}\quad\text{and}\quad \slowshift ji \left(L\cap\big(e_{h}\wedge\extalg^{k-1}V^{(h)}\big)\right)\subseteq e_{h}\wedge\extalg^{k-1}V^{(h)}.
\]
The proof now follows by counting dimensions.
\end{proof}

If $L$ is monomial with respect to $e_{j}$, then the slow
shift $\slowshift ji$ respects the monomiality in a weaker manner.
\begin{lem}
\label{lem:SlowShiftMonomialAtj}Let $L$ be a subspace of $\extalg^{k}V$
and $j\leq n$.  Let $K$ be the largest subspace of $\extalg^{k-1}V^{(j)}$ such that $e_j \wedge K \subseteq L$. If $L$ is monomial with respect to $e_{j}$, 
and if $1\leq i<j$, then 
\[
\slowshift ji L=\left(\big(L\cap\extalg^{k}V^{(j)}\big)+\left(e_{i}\wedge K\right)\right)\oplus\left(e_{j}\wedge K'\right),
\]
 where $K'$ is the largest subspace of $K$ such that $e_{i}\wedge K'\subseteq L$.
\end{lem}

\begin{proof}
It is immediate that $L\cap\extalg^{k}V^{(j)}$ is fixed under $\slowshift ji$,
while from definition
\[
\slowshift ji \left(e_{j}\wedge K\right)=\left(e_{i}\wedge K\right)\oplus\left(e_{j}\wedge\left\langle x\in K:e_{i}\wedge x=0\right\rangle \right).
\]
More broadly, for each $y\in K$ so that $e_{i}\wedge y\in L$, we
have $\slowshift ji \left(e_{j}\wedge y-e_{i}\wedge y\right)=e_{j}\wedge y$.
The proof now follows by counting dimensions.
\end{proof}
\begin{cor}
If $L$ is monomial with respect to $e_{j}$, then for each $e_{j}\wedge x$
in $\slowshift ji L$, it also holds that $e_{i}\wedge x$ is in $\slowshift ji L$. 
\end{cor}

The following will be useful as a base case for an inductive argument
that Algorithm~\ref{alg:ShiftingAlg} terminates.
\begin{cor}
\label{cor:RepeatedSlowShift}For any $1\leq i<j\leq n$ and subspace
$L$ of $\extalg^{k}V$, we have $$\slowshift ji \slowshift ji \slowshift ji L=\slowshift ji \slowshift ji L.$$
Moreover, if $L$ is monomial with respect to $e_{j}$, then $\slowshift ji \slowshift ji L= \slowshift ji L$.
\end{cor}

We caution that there are (non-monomial) examples where $\slowshift ji \slowshift ji L\neq \slowshift ji L$.
Consider the following:
\begin{example}
The span of $e_{1}\wedge e_{2}\wedge e_{3}-e_{1}\wedge e_{2}\wedge e_{4}$
requires two applications of $\slowshift 43$ to stabilize. The first
application yields the span of $e_{1}\wedge e_{2}\wedge e_{4}$. Although
this is monomial, it is not stable under $\slowshift 43$. Performing
$\slowshift 43$ again yields $e_{1}\wedge e_{2}\wedge e_{3}$, which
is stable under further slow shifts.
\end{example}

Although it is clear that the shifting algorithm for set systems
terminates, the
analogue for slow shifting of exterior subspaces requires a slightly more careful analysis.
\begin{thm}
\label{thm:ShiftingTerminates} Algorithm~\ref{alg:ShiftingAlg}
terminates, for any choice of a sequence of non-fixing slow shifts.
\end{thm}

\begin{proof}
We work by induction on the size of the set of permissible indices
$I$ given as input to the algorithm. If $\left|I\right|=1$, then
the result is trivial, and if $\left|I\right|=2$, then the result
follows from Corollary~\ref{cor:RepeatedSlowShift}. 

If $\left|I\right|>2$, then let $b$ be the greatest index in $I$.
If there is no $\slowshift bi$ operation in the sequence, then by induction
(examining $I\setminus\{b\}$) the algorithm terminates. The first
$\slowshift bi$ operation makes $L$ monomial with respect to $e_{b}$.
By Lemma~\ref{lem:SlowShiftMonomialAtj}, each additional non-fixing
$\slowshift bi$ operation reduces the dimension of the subspace consisting
of multiples of $e_{b}$. By Lemma~\ref{lem:SlowShiftMonomialAwayfromj},
each other slow shift operation preserves the dimension of this subspace.
Thus, there are finitely many non-fixing $\slowshift bi$ operations.
The result now follows by induction on $I\setminus\{b\}$.
\end{proof}
As in the set system case \cite[Proposition 2.2]{Frankl:1987}, it
is efficient to shift first from the greatest indices.  Here we base our lexicographic order on the usual order on $[n]$.
\begin{prop}
\label{prop:ShiftingTerminatesEff}If we choose at each step of Algorithm~\ref{alg:ShiftingAlg}
the lexicographically last ordered pair $(j,i)$ where $\slowshift ji$ is non-fixing,
then the algorithm terminates in at most $\left|I\right|-1+\binom{\left|I\right|}{2}$
iterations.
\end{prop}

\begin{proof}
Let $b$ be the last element of $I$. We first apply $\slowshift bi$
over all $i<b$, in decreasing order, performing the first such shift
twice if possible (and so necessary). After the first slow shift,
we have a subspace $L$ that is monomial with respect to $e_{b}$.
Additional slow shifts preserve monomiality with respect to $e_{b}$
by Lemma~\ref{lem:SlowShiftMonomialAtj}. Moreover, since $L\cap\extalg^{k}V^{(b)}\subseteq \slowshift bi \left(L\cap\extalg^{k}V^{(b)}\right)$,
each $\slowshift bi$ is non-fixing for at most one application.

Suppose that we have completed all slow shifts $\slowshift bi$ on $L$,
and let $e_{b}\wedge K$ be as in Lemma~\ref{lem:SlowShiftMonomialAtj}.
Repeated application of Lemma~\ref{lem:SlowShiftMonomialAtj} then
gives that $e_{i}\wedge K\subseteq L$ for each $i<b$ in $I$, hence
that $\bigoplus_{i\in I}e_{i}\wedge K\subseteq L$. Now by another
application of Lemma~\ref{lem:SlowShiftMonomialAtj}, the subspace
$\bigoplus_{i\in I}e_{i}\wedge K$ is preserved under all further
$\slowshift ji$ operations with $i,j\in I$. In particular, $\slowshift bi$
will fix $L$ for all $i$ throughout the remainder of the slow
shifting procedure.

That the process terminates now follows by induction on $I\setminus\{b\}$.
In the worst case we shift at each pair $i<j$, and also one extra
time at each element of $I \setminus \min I$ in order to guarantee monomiality.
 Thus, we have at most $\left|I\right|-1+\binom{\left|I\right|}{2}$ slow
shifts in this process.
\end{proof}
\begin{rem}
It is instructive to compare and contrast with \cite[Proposition 2.2]{Frankl:1987},
which applies a similar sequence of combinatorial shifting operations
to sets in the case where $I=[n]$. The set situation requires only $\binom{n}{2}$
operations, since it is not required to first make the system monomial
with respect to each $i$.
\end{rem}

By Theorem~\ref{thm:ShiftingTerminates} and Proposition~\ref{prop:ShiftingTerminatesEff},
given a subspace $L$ of $\extalg^{k}V$, we may find a subspace of
the same dimension that is stable under $\slowshift ji$ over all $i<j$
both in $I$. We now describe the resulting subspaces.

\begin{prop}
\label{prop:StableShifted}Let $L$ be a subspace of $\extalg^{k}V$,
and let $I\subseteq[n]$. If $L$ is stable under slow shifting $\slowshift ji$
over all $i<j$ in $I$, then for each $x\in\extalg V^{(I)}$, the variable sets
of the monomials $y$ in $\extalg V^{([n] \setminus I)}$ so that $x\wedge y\in L$
form a shifted set system.
\end{prop}

\begin{proof}
Let $x \wedge y$ be in the hypothesis.  By definition, $\slowshift ji x \wedge y = x \wedge \slowshift ji y$.  As we observed in Section~\ref{subsec:SlowShift} that $\slowshift ji$ acts on monomials as combinatorial shifting on their variable sets, the result follows. 
\end{proof}

\begin{thm}
\label{thm:MonomialBasisAfterShifting}Let $L$ be a subspace of $\extalg^{k}V$,
 let $I\subseteq[n]$, and let $a$ be $\min I$. If $L$ is stable under slow shifting $\slowshift ji$
over all $i<j$ in $I$, then $L$ has a basis consisting of elements of
the form $x\wedge y$, where $x$ is a homogeneous form in $\extalg V^{(I\setminus a)}$
and $y$ is a monomial whose variable set is a subset of $\left\{ e_{h}:h\in I\right\} $.
\end{thm}

\begin{proof}
Immediate from repeated application of Lemmas~\ref{lem:SlowShiftMakesMonomial}
and \ref{lem:SlowShiftMonomialAwayfromj}, together with basic facts
about direct sums.
\end{proof}
Thus, Theorem~\ref{thm:MonomialBasisAfterShifting} says that if
$L$ is stable under $\slowshift ji$ over $I$, then $L$ is simultaneously
monomial with respect to all variables with index in $I$ except possibly for the least indexed.
In certain circumstances, we can also get monomiality with respect
to the least index of $I$.
\begin{cor}
\label{cor:SlowShiftingBorelFixed}Let $L$ be a subspace of $\extalg^{k}V$.
If $L$ is stable under slow shifting $\slowshift ji$ over all  $i,j\in[n]$,
then $L$ has a monomial basis.
\end{cor}

\begin{proof}
Up to constant multiplication, the only homogeneous forms in $\extalg V^{(\{2,\dots,n\})}$
are the $0$-form $1 = e_\emptyset$ and the $1$-form $e_{1}$.
\end{proof}

We remark that if $L$ is stable under $\slowshift ji$ over all $i<j$ in
$I$, but is not monomial with respect to $a = \min I$, then it is easy
to see that applying $M_{a}$ (as in Section~\ref{subsec:MinimizeVarAction}) will result in a system that is monomial with respect to all indices in $I$. Indeed, one might see $M_{a}$ as a ``degenerate''
slow shifting operation, in the sense that we send $e_{a}$ to
zero (as in Section~\ref{subsec:LimitActions}), but do
not have a lesser indexed variable available with which to replace it.

An essential step of our main proof will use shifting over $I = \{ 3,\dots,n \}$.  

\begin{cor}
\label{cor:MonomialityAllBut2}Let $L$ be a subspace of $\extalg^{k}V$
that contains $e_{1}\wedge e_{2}\wedge\extalg^{k-2}V$ and that is
annihilated by $e_{1}\wedge e_{2}$. If $L$ is stable under slow
shifting $\slowshift ji$ over all $i,j\geq3$, then $L$ has a basis consisting
of elements that are monomial with respect to all $e_{j}$ with $j\geq3$.
\end{cor}

\begin{proof}
In the basis yielded by Theorem~\ref{thm:MonomialBasisAfterShifting},
the form $x$ has exterior degree $0,1,2,$ or $3$. Without loss
of generality, the basis contains the monomials of $e_{1}\wedge e_{2}\wedge\extalg^{k-2}V$.
It is immediate that for degree $0$ or $3$, we have the desired.
For degree $1$, the annihilation condition gives that $x=\lambda_{1}e_{1}+\lambda_{2}e_{2}$.
For degree $2$, we eliminate $e_{1}\wedge e_{2}$ using the containment
condition, leaving $x=\lambda_{1}e_{1}\wedge e_{3}+\lambda_{2}e_{2}\wedge e_{3}$.
\end{proof}

It seems worth noting that, in the case $I = \{ 3, \dots, n \}$, a subspace stable under all $\slowshift ji$'s decomposes in a pleasing manner.

\begin{cor}
Let $L$ be a subspace of $\extalg^{k}V$ that contains $e_{1}\wedge e_{2}\wedge\extalg^{k-2}V$
and that is annihilated by $e_{1}\wedge e_{2}$. If $L$ is stable
under slow shifting $\slowshift ji$ over all $i,j\geq3$, then there are
pairwise linearly independent vectors $w_{1},\dots,w_{m}$ in the
span of $e_{1},e_{2}$ and subspaces $K_{1},\dots,K_{m}$ of $\extalg^{k-1}V^{(\{1,2\})}$
so that 
\[
L=e_{1}\wedge e_{2}\wedge\extalg^{k-2}V+\sum w_{i}\wedge K_{i},
\]
and with the following additional properties:
\begin{enumerate}
\item For every distinct $i,j\in[m]$, we have $K_{i}\cap K_{j}=\bigcap_{h}K_{h}$.
\item Each $K_{i}$ as well as $\bigcap_{h}K_{h}$ has a basis of
monomials whose variable sets form a shifted set system.
\end{enumerate}
\end{cor}

\begin{proof}
Take $K_{i}$ to be maximal under the condition that $w_i \wedge K_i \subseteq L$ and apply Corollary~\ref{cor:MonomialityAllBut2}.  Suppose $x\in K_{i}\cap K_{j}$. Then $w_{i}\wedge x$ and $w_{j}\wedge x$
are in $L$. Thus, arbitrary combinations of $w_{i}\wedge x$ and
$w_{j}\wedge x$ are in $L$. In particular, as the span of $w_i$ and $w_j$ is the same as the span of $e_1$ and $e_2$, it holds that $w_{h}\wedge x$ is
in $L$. Monomiality and shiftedness follows from Corollary~\ref{cor:MonomialityAllBut2} and Proposition~\ref{prop:StableShifted}.
\end{proof}

\section{\label{sec:CrossAnn}Cross-annihilating subspaces}

In this section, we prove Theorem~\ref{thm:CrossAnnihilating-ext}.
With the framework that we have developed, the proof is easy. By Lemma~\ref{lem:LimitsPreserveCrossAnn},
slow shifting preserves the cross-annihilating property. Applying
the slow shifting procedure over $I=[n]$, by Theorem~\ref{thm:ShiftingTerminates},
Corollary~\ref{cor:SlowShiftingBorelFixed} and Proposition~\ref{prop:StableShifted}
(with $x=1$), we can reduce to the case where $K$ and $L$ have
a monomial basis. The variable sets of the monomials in the bases form
shifted, nonempty, cross-intersecting set systems. The result now
follows by \cite[Theorem 2]{Hilton/Milner:1967}.

\section{\label{sec:HMproof}Nontrivially self-annihilating subspaces}

In this section, we prove Theorem~\ref{thm:HMext}. 

We use a similar proof strategy as in the proof of the set system result.  We briefly sketch this proof: one reduces to the case where either
$1$ or $2$ is in every set in the set system, then apply a cross-intersecting
bound. 
A difficulty in extending this approach to the exterior algebra case is that
an element of $L$ may have both $e_{1}$ and $e_{2}$ occurring in its support,
but not have $e_{1}\wedge e_{2}$ as a factor. 

\subsection{Main lemma}

The following technical lemma says that Theorem~\ref{thm:HMext} holds under the extra conditions of
annihilation by a $2$-form and partial annihilation by a $1$-form
factor.
\begin{lem}
\label{lem:HMext-2ann}Let $k\leq n/2$, and let $f,g\in V$. If $L$
is a nontrivially self-annihilating subspace of $\extalg^{k}V$, such
that $L$ is annihilated by $f\wedge g$, and where some element
of $L$ is annihilated by $f$ but not by $g$, then $\dim L\leq\binom{n-1}{k-1}-\binom{n-k-1}{k-1}+1$.
\end{lem}

\begin{proof}
Let $V_{0} \subseteq V$ be a complementary subspace to the span
of $\{f,g\}$. Without loss of generality, it holds that $A=f\wedge g\wedge\extalg^{k-2}V$
is contained in $L$, as otherwise $A+L$ satisfies the hypothesis.

We now split $L$ into the direct sum of three subspaces: the subspace
$A$ as in the previous paragraph, the subspace $B$ of elements that are annihilated
by $f$ but not by $g$, and a complementary subspace $C$. Moreover,
$B$ consists of elements of the form $f\wedge x$, and we can choose
$C$ to consist of elements of the form $f\wedge x+g\wedge y$, where
$x,y\in\extalg^{k-1}V_{0}$ and $y\neq0$. 

Now we take $B'\subseteq\extalg^{k-1}V_{0}$ to be $\{x:f\wedge x\in B\}$,
and $C'\subseteq\extalg^{k-1}V_{0}$ to be 
\[
\left\{ y:f\wedge x+g\wedge y\in C\text{ for some }x\right\} .
\]
Since no element of $C$ is annihilated by $f$, we get that $\dim C'=\dim C$.
Since $B$ and $C$ are cross-annihilating, and by the definition
of exterior multiplication, we obtain that $B'$ and $C'$ are cross-annihilating.

But now by Theorem~\ref{thm:CrossAnnihilating-ext}, we see that 
\[
\dim L=\dim A+\dim B'+\dim C'\leq\binom{n-2}{k-2}+\binom{n-2}{k-1}-\binom{n-k-1}{k-1}+1.
\]
The result now follows by the Pascal's triangle identity.
\end{proof}

\subsection{Other lemmas}

We also need several lemmas that parallel results in combinatorial
set theory.

The first is immediate by Lemma~\ref{lem:LimitsPreserveCrossAnn},
and parallels that if $1$ or $2$ is in every set $F\in\mathcal{F}$, then
the same holds after combinatorial shifting.
\begin{lem}
\label{lem:Ann2formPreserved}If $L\subseteq\extalg^{k}V$ is a subspace
so that $e_{1}\wedge e_{2}\wedge L=0$, then also $e_{1}\wedge e_{2}\wedge(\slowshift ji L)=0$
for every $3\leq i<j\leq n$.
\end{lem}

The second is immediate by computation. 
\begin{lem}
\label{lem:Ann2formto1form}Let $f,g,\ell$ be $1$-forms of $\extalg V$,
such that $f$ and $g$ are linearly independent. If $L$ is a subspace
of $\extalg^{k}V$ so that $f\wedge g\wedge\extalg^{k-2}V\subseteq L$,
and so that $\ell$ annihilates $L$, then $\ell$ is in the span
of $\{f,g\}$.
\end{lem}

The third is non-trivial, and requires analysis of the slow shifting
operation.
\begin{lem}
\label{lem:Technical2div}Let $\ell$ be a $1$-form of $\extalg V$.
If $L$ is a subspace of $\extalg^{k}V$ such that $\ell\wedge \slowshift ji L=0$,
then $\ell\wedge(e_{i}-e_{j})\wedge L=0$. 

Moreover, if $L$ is nontrivial, then $\ell\wedge(e_{i}-e_{j})\neq0$.
\end{lem}

Broadly speaking, Lemma~\ref{lem:Technical2div} is an extension
of the fact in combinatorial set theory that if $\shift ji\mathcal{F}$
becomes trivial, then every set in $\mathcal{F}$ contains at least one
of $i,j$.
\begin{proof}[Proof (of Lemma~\ref{lem:Technical2div}).]
 By (\ref{eq:SlowShiftExtalg}), whenever $x+e_{j}\wedge y$ is in
$L$ (where the variable sets of the supports of $x$ and $y$ do not contain $e_{j}$),
we have $\ell\wedge(x+e_{i}\wedge y)=0$. Thus,
\begin{align*}
\ell\wedge(e_{i}-e_{j})\wedge(x+e_{j}\wedge y) & =\ell\wedge e_{i}\wedge x-\ell\wedge e_{j}\wedge(x+e_{i}\wedge y)\\
 & =\ell\wedge e_{i}\wedge x.
\end{align*}
 Since $\ell\wedge(x+e_{i}\wedge y)=0$, we have also $\ell\wedge e_{i}\wedge x=\ell\wedge e_{i}\wedge(x+e_{i}\wedge y)=0$.

Finally, if $L$ is nontrivial, then there is an $m=x+e_{j}\wedge y$
that is not annihilated by $e_{i}-e_{j}$. Thus, $e_{i}\wedge x-e_{j}\wedge(x+e_{i}\wedge y)\neq0$,
and in particular $x\neq-e_{i}\wedge y$. But then (\ref{eq:SlowShiftExtalg}) gives that $\slowshift ji m = x + e_i \wedge y$. It now follows by computation
that $(e_{i}-e_{j})\wedge m=(e_{i}-e_{j})\wedge \slowshift ji m$, so
that the latter is nonzero. This yields that $\ell$ and $e_{i}-e_{j}$
are linearly independent, as desired.
\end{proof}

\subsection{Proof of Theorem~\ref{thm:HMext}}

Let $L$ be a nontrivially self-annihilating subspace of maximal dimension.  As the theorem is trivial for $k=1$, assume that $k\geq 2$.
Apply slow shifting operations over all $i,j\in[n]$ and with respect
to the standard basis $\mathbf{e}$ until either $L$ stabilizes or is annihilated
by a $1$-form.

If $L$ stabilizes to a nontrivially self-annihilating subspace under
these slow shifting operations, then by Corollary~\ref{cor:SlowShiftingBorelFixed},
the resulting $L$ is monomial with respect to $\mathbf{e}$, and the desired bound follows
from (the shifted version of) Hilton-Milner for set systems, Theorem~\ref{thm:HM}.

If $\slowshift ji L$ is annihilated by a 1-form $\ell$, then we change
to the ordered basis $\mathbf{f}$ with $f_{1}=\ell$, $f_{2}=e_{j}-e_{i}$,
and $f_{3},\dots,f_{n}$ chosen arbitrarily to fill out the basis.
It follows from Lemma~\ref{lem:Technical2div} that $f_{1}$ and
$f_{2}$ are indeed linearly independent, and that $L$ is annihilated
by $f_{1}\wedge f_{2}$. Apply slow shifting operations with respect
to the basis $\mathbf{f}$ for $i,j\in\{3,\dots,n\}$ until either
$L$ stabilizes or is annihilated by a $1$-form. We may assume by
maximality that $f_{1} \wedge f_{2} \wedge \extalg^{k-2} V \subseteq L$,
and notice that Lemma~\ref{lem:Ann2formPreserved} gives this subspace
to be preserved~by the slow shifting.

If $L$ stabilizes, then by Corollary~\ref{cor:MonomialityAllBut2},
we may find a basis consisting of elements of the form $f_{1}\wedge f_{2}\wedge x$
and $(\lambda_1 f_{1} + \lambda_2 f_{2})\wedge y$. In particular, the resulting
$L$ satisfies the conditions of Lemma~\ref{lem:HMext-2ann} for
some $f,g$ in the span of $\{f_{1},f_{2}\}$. The desired bound follows.

Finally, if $\slowshift ji L$ is annihilated by a 1-form $\ell'$ for some $3\leq i<j$,
then Lemma~\ref{lem:Technical2div} gives that $\ell'\wedge(f_{j}-f_{i})$
annihilates $L$. Lemma~\ref{lem:Ann2formto1form} gives that $\ell'$
is in $\spanop\{f_{1},f_{2}\}$; take $g$ to be such that $\ell'\wedge g=f_{1}\wedge f_{2}$.
By maximality, $\ell'\wedge(f_{j}-f_{i})\wedge\extalg^{k-2}V$ is
contained in $L$.

Now $f_{1}\wedge f_{2}\wedge\extalg^{k-2}V=\ell'\wedge g\wedge\extalg^{k-2}V\subseteq L$,
and (since $g$ is in the span of $\{f_{1},f_{2}\}$) there are elements
of $\ell'\wedge(f_{j}-f_{i})\wedge\extalg^{k-2}V$ that are annihilated
by $\ell'$ but not by $g$. Thus, we may apply Lemma~\ref{lem:HMext-2ann}
to finish the proof.

\bibliographystyle{hamsplain}
\vspace*{-.03cm}
\bibliography{main}

\end{document}